\documentclass[12pt,reqno]{amsart}
\usepackage[T1]{fontenc}
\usepackage{amsfonts}
\usepackage{amssymb}
\usepackage{mathrsfs}
\usepackage[latin1]{inputenc}
\usepackage{amsmath}
\usepackage[english]{babel}
\usepackage{setspace}
\usepackage{latexsym,amsfonts,amsmath,amssymb,soul}

 \newtheorem{thm}{Theorem}[section]
 \newtheorem{cor}[thm]{Corollary}
 \newtheorem{lem}[thm]{Lemma}
 \newtheorem{prop}[thm]{Proposition}
 \theoremstyle{definition}
 \newtheorem{defn}[thm]{Definition}
 \newtheorem{rem}[thm]{Remark}

 \numberwithin{equation}{section}

 \numberwithin{equation}{section}

\newcommand{\R}{{\mathbb R}}

\newcommand{\C}{{\mathbb C}}
\newcommand{\N}{{\mathbb N}}

\newcommand{\cD}{{\mathcal D}}
\newcommand{\cL}{{\mathcal L}}
\newcommand{\cF}{{\mathcal F}}
\newcommand{\cE}{{\mathcal E}}

\newcommand{\cS}{{\mathcal S}}
\newcommand{\cO}{{\mathcal O}}

\newcommand{\su}{\subseteq}

\newcommand\proj{\mathop{\rm proj\,}}
\newcommand\ind{\mathop{\rm ind\,}}

\begin{document}


\title[Multiplication and convolution topological algebras ]
 {Multiplication and convolution topological algebras in spaces of $\omega$-ultradifferentiable functions of Beurling type}

%----------Author 1
\author[A.A. Albanese, C. Mele]{Angela\,A. Albanese and Claudio Mele}

\address{ Angela A. Albanese\\
Dipartimento di Matematica e Fisica ``E. De Giorgi''\\
Universit\`a del Salento- C.P.193\\
I-73100 Lecce, Italy}
\email{angela.albanese@unisalento.it}

\address{Claudio Mele\\
	Dipartimento di Matematica e Fisica ``E. De Giorgi''\\
	Universit\`a del Salento- C.P.193\\
	I-73100 Lecce, Italy}
\email{claudio.mele1@unisalento.it}

%\thanks{\textsuperscript{*} Corresponding author}

%\subjclass{Primary 47A10, 47B38, 47D06; Secondary 46A04, 46E10, 47A16, 47A35.}

%\keywords{Ces\`aro operator, Fr\'echet space of smooth functions, spectrum, mean ergodic operator, $C_0$-semigroup.}

%\date{26/05/2015}
\dedicatory{Dedicated to Prof. F. Altomare on the occasion of his 70th birthday }
%%% on the occasion----------------------------------------------------------------------

\begin{abstract}
We determine multiplication and convolution topological algebras for classes of $\omega$-ultradifferentiable functions of Beurling type. Hypocontinuity and discontinuity of the multiplication and convolution  mappings are also investigated.
% It is also proved that
% the space of multipliers $\cO_{M,\omega}(\R^N)$ and 
% the pre-dual  $\cO_{C,\omega}(\R^N)$ of the space of convolutors  of $\cS_\omega(\R^N)$ is a sequentially retractive  Montel (LF)-space.
\end{abstract}

%%% ----------------------------------------------------------------------
\maketitle
%%% ----------------------------------------------------------------------
%\tableofcontents
\section{Introduction }\label{intro}
%Bj\"orck \cite{B} introduced in 1965 global classes of ultradifferentiable rapidly decreasing function spaces $\cS_{\omega}(\R^N)$, using weights $\omega$ in the sense of Beurling \cite{Be} to extend previous theorems of H\"ormander about interior regularity of linear partial differential operators with constant coefficients. These weight functions permit to treat in a unified way a big scale of classes of functions or (ultra)distributions (see Komatsu \cite{K}) and are especially suitable for manipulations on the Fourier transform side. Braun, Meise and Taylor \cite{BMT} showed that these classes can be defined by the decay behaviour of their derivatives by using the Young conjugate of the function $t \mapsto \omega(e^t)$.

Schwartz started in 1966 the study of multipliers and convolutors of the space $\cS(\R)$ of rapidly decreasing functions. The interest lies in the importance of their application to the study of partial differential equations. Since then many authors introduced and  studied particular aspects of the spaces of multipliers and of convolutors for ultradifferentiable classes of rapidly decreasing functions of Beurling or Roumieu type in the sense of Komatsu \cite{K} (see  \cite{De,De1,De2,De3,DD} for recent results in this setting). In the last years  the attention  has focused on the study of the space $\cS_\omega(\R^N)$ of the ultradifferentiable rapidly decreasing functions of Beurling type, as introduced by Bj\"orck \cite{B} (see
\cite{BJO,BJOR,BJOS,De1}, for instance, and the references therein).  Inspired by this line of research and by the previous work, in \cite{AC,AC2} the authors introduced and studied the space $\cO_{M,\omega}(\R^N)$ of the slowly increasing functions of Beurling type in the setting of ultradifferentiable function spaces of Beurling type, showing that it is the space of the multipliers of the space $\cS_\omega(\R^N)$ and of its dual $\cS'_\omega(\R^N)$, and the space $\cO_{C,\omega}(\R^N)$ of the very slowly increasing functions of Beurling type, whose strong dual $\cO'_{C,\omega}(\R^N)$ is the space of the convolutors of the space $\cS_\omega(\R^N)$ and of its dual $\cS'_\omega(\R^N)$. Their rich topological structure led us to determine deeply results concerning regularity, equivalent systems of seminorms, representations of the ultradistributions and the action of the Fourier transform, that is a topological isomorphism from $\cO_{M,\omega}(\R^N)$ to $\cO'_{C,\omega}(\R^N)$.
% when the former space is endowed with the strong operator lc-topology induced by $\mathcal{L}_b(\cS_{\omega}(\R^N))$ and the latter space is endowed with its natural lc-topology.

In this paper the authors continue the study of these spaces to complete their description. 
%One of the first aim is to  show that $\cO_{C,\omega}(\R^N)$ is a sequentially retractive Montel (LF)-space.
% We point out that the methods of the proofs are different from the ones used in \cite{De,De1,De2}, that relies on tools from the time-frequency analysis as the short time Fourier transform (STFT). 
The  aim  is to establish that $(\cO_{M,\omega}(\R^N),\cdot)$, $(\cS_{\omega}(\R^N),\cdot)$ are multiplication topological algebras, while $(\cS_\omega(\R^N),\star)$ and $(\cO'_{C,\omega}(\R^N),\star)$ are convolution topological algebras. We also determine the spaces of multipliers and of convolutors of the spaces $\cO_{M,\omega}(\R^N)$, $\cO_{C,\omega}(\R^N)$ and their duals.
 Furthermore, we  analyze  the continuity of the  multiplication and convolution bilinear mappings on some pairs between the spaces $\cO_{M,\omega}(\R^N)$, $\cO_{C,\omega}(\R^N)$, $\cS_\omega(\R^N)$ and their duals, studied classically by Schwartz \cite{S}  (see also Larcher \cite{La} and the references therein). This approach is done not only treating the hypocontinuity,
 but also describing and investigating  the continuity properties of these mappings.

The paper is organized as follows. Section 2 is devoted to recalling some  definitions and properties of the weights $\omega$ and of the $\omega$-ultradifferentiable functions, that we use in the following. 
%In particular, we show the sequential retractivity of the space $\cO_{C,\omega}(\R^N)$.
 In Section 3 we prove the results about the multiplication and convolution (topological) algebras and determine the spaces of multipliers and of convolutors of the spaces under consideration. Finally, in the last section, we analyze the multiplication and convolution bilinear mappings on some pairs between the spaces $\cO_{M,\omega}(\R^N)$, $\cO_{C,\omega}(\R^N)$, $\cS_{\omega}(\R^N)$ and their duals, proving when are hypocontinuous or discontinuous.

\section{Definitions and preliminary results}

%%%%%%%%%%%%%%%%%%%%%%%%%%%%%%%
We first give the definition of non-quasianalytic weight function in the sense of Braun, Meise and Taylor \cite{BMT} suitable
for the Beurling case, i.e., we also consider the logarithm as a weight function.

\begin{defn}\label{D.weight} A non-quasianalytic weight function is a continuous increasing function $\omega:[0,\infty)\to[0,\infty)$  satisfying the following properties:
	\begin{itemize}
		\item[($\alpha$)] there exists $K\geq1$ such that $\omega(2t)\leq K(1+\omega(t))$ for every $t\geq0$;
		\item[($\beta$)] $\int_{1}^{\infty} \frac{\omega(t)}{1+t^2}\, dt < \infty$;
		\item[($\gamma$)] there exist $a \in \mathbb{R}$, $b>0$ such that $\omega(t) \geq a+b\log(1+t)$, for every $t\geq0$;
		\item[($\delta$)] $\varphi_\omega(t)= \omega \circ \exp(t)$ is a convex function.
	\end{itemize}
\end{defn}

	We recall some known properties of the weight functions that shall be useful in the following (the proofs can be found in the literature):

(1) Condition $(\alpha)$ implies that
\begin{equation}\label{sub}
\omega(t_1+t_2)\leq K(1+\omega(t_1)+\omega(t_2)), \; \forall t_1, t_2 \geq 0.
\end{equation}
Observe that this condition is weaker than subadditivity (i.e.,  $\omega(t_1+t_2)\leq \omega(t_1) + \omega(t_2))$. The weight functions satisfying ($\alpha$) are not necessarily subadditive in general.

(2) Condition $(\alpha)$ implies that there exists $L\geq 1$ such that
\begin{equation}\label{l}
\omega(e t)\leq L(1+\omega(t)), \; \forall t \geq 0.
\end{equation}

(3) By condition $(\gamma)$ we have  that
\begin{equation}\label{eq.Lpspazi}
e^{-\lambda \omega(t)}\in L^p(\R^N), \; \forall  \lambda\geq \frac{N+1}{bp}.
\end{equation}

Given a non-quasianalytic weight function $\omega$, we define the Young conjugate $\varphi^*_\omega$ of $\varphi_\omega$ as the function $\varphi^*_\omega:[0,\infty)\to [0,\infty)$ by
\begin{equation*}\label{Yconj}
\varphi^*_\omega(s):=\sup_{t\geq 0}\{ st-\varphi_\omega(t)\},\quad s\geq 0.
\end{equation*}
There is no loss of generality to assume that $\omega$ vanishes on $[0,1]$. Therefore, $\varphi^*_\omega$ is convex and increasing, $\varphi^*_\omega(0)=0$ and $(\varphi^*_\omega)^*=\varphi_\omega$. 

Further useful properties of $\varphi^*_\omega$  are listed below (see \cite{BMT}):
\begin{enumerate}
	\item[\rm (1)] $\frac{\varphi^*_\omega(t)}{t}$ is an increasing function in $(0,\infty)$.
	\item[\rm (2)] For every $s,t\geq 0$ and $\lambda>0$
	\begin{equation}\label{secondprop}
	2\lambda\varphi^*_\omega\left(\frac{s+t}{2\lambda}\right)\leq \lambda \varphi^*_\omega\left(\frac {s}{\lambda}\right)+\lambda\varphi^*_\omega\left(\frac{t}{\lambda}\right) \leq \lambda\varphi^*_\omega\left(\frac{s+t}{\lambda}\right).
	\end{equation}
	\item[\rm (3)] For every $t\geq 0$ and $\lambda>0$
	\begin{equation*}\label{eq.bb}
	\lambda L \varphi^*_\omega\left(\frac{t}{\lambda L}\right)+t\leq \lambda \varphi^*_\omega\left(\frac{t}{\lambda}\right)+\lambda L,
	\end{equation*}
	where $L\geq 1$ is the constant appearing in formula $(\ref{l})$.
	\item[\rm (4)] For all $m,M\in\mathbb{N}$ with $M\geq mL$, where $L$ is the constant appearing in formula $(\ref{l})$, and for every $t\geq 0$
	\begin{equation}\label{firstprop}
	2^{t}\exp\left(M\varphi^*_\omega\left(\frac{t}{M}\right)\right)\leq C \exp\left(m\varphi^*_\omega\left(\frac{t}{m}\right)\right),
	\end{equation}
	with $C:=e^{mL}$.
\end{enumerate}

We now introduce the ultradifferentiable function space $\cS_\omega(\R^N)$ in the sense of Bj\"ork \cite{B}.

\begin{defn}\label{D.Beurling} Let $\omega$ be a non-quasianalytic weight function. We denote by  $\mathcal{S}_\omega(\mathbb{R}^N)$  the set of all functions $f\in L^1(\mathbb{R}^N)$ such that $f,\hat{f}\in C^\infty(\mathbb{R}^N)$ and for all $\lambda >0$ and $\alpha\in\N_0^N$ we have
	\begin{equation*}\label{cond Sw 1}
	\| \exp(\lambda\omega)\partial^\alpha f\|_\infty <\infty\ \ {\rm and }\ \  
	\|\exp(\lambda\omega)\partial^\alpha\hat{f}\|_\infty <\infty \; ,
	\end{equation*}
	where $\hat{f}$ denotes the Fourier transform of $f$. The elements of  $\mathcal{S}_\omega(\mathbb{R}^N)$ are called \textit{$\omega$-ultradifferentiable rapidly decreasing functions of Beurling type}. 
	We denote by  $\cS'_{\omega}(\R^N)$ the dual of $\cS_{\omega}(\R^N)$  endowed with its strong topology.
\end{defn}

The space $\cS_\omega(\R^N)$ is a Fr\'echet space with different equivalent systems of seminorms (see \cite[Theorem 4.8]{BJOR} and \cite[Theorem 2.6]{BJO}). In the following, we will use the following  system of norms generating the Fr\'echet topology
of $\cS_\omega(\R^N)$: 
	\[
q_{\lambda,\mu}(f):=\underset{\alpha \in \mathbb{N}^N_0}{\sup}\| \exp(\mu\omega) \partial^\alpha f\|_\infty\exp\left(-\lambda\varphi^*_\omega\left(\frac{|\alpha|}{\lambda}\right)\right), \quad \lambda,\mu>0,\ f\in \cS_\omega(\R^N),\]
or equivalently, the sequence of norms $\{q_{m,n}\}_{m,n\in\N}$.

We point out that the space $\cS_\omega(\R^N)$ is a nuclear Fr\'echet space, see, f.i., \cite[Theorem 3.3]{BJOS} or \cite[Theorem 1.1]{De1}.

We refer to  \cite{BMT} for the definition and  the main properties of the ultradifferentiable function spaces $\cE_\omega(\Omega)$, $\cD_\omega(\Omega)$ and their duals of Beurling type in the sense of Braun, Meise and Taylor. We only recall that 
for an open subset $\Omega$ of $\R^N$, the space $	\cE_\omega(\Omega)$ is defined as
	\[
	\cE_\omega(\Omega):=\left\{f\in C^\infty(\Omega)\colon p_{K,m}(f)<\infty\  \forall K \Subset \Omega,\, m\in\N  \right\},
	\]
	where
	\[
	 p_{K,m}(f):=\sup_{x\in K}\sup_{\alpha\in\N_0^N}|\partial^{\alpha}f(x)|\exp\left(-m \varphi^*_\omega\left(\frac{|\alpha|}{m}\right)\right).
	\]
 $	\cE_\omega(\Omega)$ is a nuclear Fr\'echet space with respect to  the lc-topology generated by the system of seminorms $\{p_{K,m}\}_{K\Subset \Omega, m\in\N}$ (see \cite[Proposition 4.9]{BMT}).
	The elements of $\cE_\omega(\Omega)$ are called \textit{$\omega$-ultradifferentiable functions
		of Beurling  type} on $\Omega$.

%%%%%%%%%%%%%%%%%%%%%%%%%%%%%%%%%%%%%%%%%%%%%%%%%%%%%

%%%%%%%%%%%%%%%%%%%%%%%%%%%%%%%%%%%%%%%%%%%%%%%%%%%%

The spaces  $\cO_{M,\omega}(\R^N)$ and $\cO_{C,\omega}(\R^N)$ have been introduced in \cite{AC} and the definition has been given there in terms of weighted $L^\infty$-norms as it follows.

\begin{defn}\label{D.spaziO}
	Let $\omega$ be a non-quasianalytic weight function. 
	
	(a)	The space  $\mathcal{O}_{M,\omega}(\mathbb{R}^N)$ of \textit{slowly increasing functions of Beurling type} on $\R^N$ is defined  by
\begin{equation*}\label{eq.OM}
	\cO_{M,\omega}(\R^N):=\bigcap_{m=1}^{\infty}\bigcup_{n=1}^\infty \mathcal{O}^m_{n,\omega}(\mathbb{R}^N),
	\end{equation*}
	where  
\begin{align*}
	&\mathcal{O}^m_{n,\omega}(\mathbb{R}^N):=\\
	&\left\{f\in C^\infty(\R^N)\colon 	r_{m,n}(f):= \underset{\alpha\in \mathbb{N}^N_0}{\sup}\underset{x\in\mathbb{R}^N}{\sup}\, |\partial^\alpha f(x)|\exp\left(-n\omega(x)-m\varphi^*_\omega\left(\frac{|\alpha|}{m}\right)\right)<\infty\right\}
	\end{align*}
	 endowed with the norm $r_{m,n}$ is a Banach space for any $m,n\in\N$.
	The space $\mathcal{O}_{M,\omega}(\mathbb{R}^N)$ is endowed with its natural lc-topology $t$, i.e., $\cO_{M,\omega}(\R^N)=\proj_{\stackrel{\leftarrow}{m}}\,\ind_{\stackrel{\rightarrow}{n}}\, \mathcal{O}^m_{n,\omega}(\mathbb{R}^N)$ is the projective limit of the (LB)-spaces $\cO^m_\omega(\R^N):=\ind_{\stackrel{\rightarrow}{n}}\, \mathcal{O}^m_{n,\omega}(\mathbb{R}^N)$.
	
	(c) The space  $\cO_{C,\omega}(\mathbb{R}^N)$ of \textit{very slowly increasing functions of Beurling type} on $\R^N$ is defined  by
	\begin{equation*}\label{eq.OC}
	\cO_{C,\omega}(\R^N):=\bigcup_{n=1}^{\infty}\bigcap_{m=1}^\infty \mathcal{O}^m_{n,\omega}(\mathbb{R}^N).
	\end{equation*}
	 The space $\cO_{C,\omega}(\R^N)$  endowed with its natural lc-topology is an (LF)-space, i.e., $\cO_{C,\omega}(\R^N)=\ind_{\stackrel{\rightarrow}{n}}\cO_{n,\omega}(\R^N)$, where $\cO_{n,\omega}(\R^N):=\proj_{\stackrel{\leftarrow}{m}}\, \mathcal{O}^m_{n,\omega}(\mathbb{R}^N)$.
	 
	 Let $\cO'_{M,\omega}(\R^N)$ ($\cO'_{C,\omega}(\R^N)$, resp.) denote the strong dual of $\cO_{M,\omega}(\R^N)$ (of $\cO_{C,\omega}(\R^N)$, resp.) endowed with its strong topology.
\end{defn}

The space $\mathcal{O}_{M,\omega}(\mathbb{R}^N)$ is the space of multipliers of $\cS_\omega(\R^N)$ and of its dual space $\cS'_\omega(\R^N)$ as proved in \cite{AC,De}, i.e., $f\in \cO_{M,\omega}(\R^N)$ if, and only if, $fg\in \cS_\omega(\R^N)$ for all $g\in \cS_\omega(\R^N)$ if, and only if, $fT\in \cS'_{\omega}(\R^N)$ for all $T\in \cS'_\omega(\R^N)$. Moreover, if $f\in\cO_{M,\omega}(\R^N)$, then the linear operators $M_f: \mathcal{S}_{\omega}(\mathbb{R}^N)\to\mathcal{S}_{\omega}(\mathbb{R}^N)$ defined by $M_f(g):=fg$, for  $g\in\mathcal{S}_{\omega}(\mathbb{R}^N)$, and $M_f\colon \mathcal{S}'_\omega(\mathbb{R}^N)\to  \mathcal{S}'_\omega(\mathbb{R}^N)$ defined by $M_f(T):=fT$, for $T\in \mathcal{S}'_\omega(\mathbb{R}^N)$, are continuous. Furthemore,  $\mathcal{O}_{M,\omega}(\mathbb{R}^N)$ is an ultrabornological space by \cite[Theorem 5.3]{De}. Due to  \cite[Theorem 5.3]{De} and the results in \cite{AC}, we have that $\cO_{M,\omega}(\R^N)$ is a closed subspace of $\cL_b(\cS_\omega(\R^N))$  ($\cL_b(\cS_\omega(\R^N))$ denotes the space of all continuous linear operators from $\cS_\omega(\R^N)$ into itself endowed with the topology of the uniform convergence on bounded subsets of $\cS_\omega(\R^N)$). In  particular, a fundamental system of continuous norms on $\cO_{M,\omega}(\R^N)$ is given by 
\begin{equation}\label{eq.nuovenorme}
q_{m,g}(f):=\sup_{\alpha\in \N_0^N}\sup_{x\in \R^N}|g(x)||\partial^\alpha f(x)|\exp\left(-m\varphi^*_\omega\left(\frac{|\alpha|}{m}\right)\right),\quad g\in\cS_\omega(\R^N),\ m\in\N,
\end{equation}
see \cite[Proposition 5.6 and Theorem 5.9]{AC}.

For any $\omega$ non-quasianalytic weight function satisfying the condition $\log(1+t)=o(\omega(t))$ as $t\to\infty$, the space $\mathcal{O}'_{C,\omega}(\mathbb{R}^N)$ is the space of convolutors of $\cS_\omega(\R^N)$ and of its dual space $\cS'_\omega(\R^N)$ as proved in \cite{AC2}, i.e., $T\in \cO'_{C,\omega}(\R^N)$ if, and only if, $T\star f\in \cS_\omega(\R^N)$ for all $f\in \cS_\omega(\R^N)$ if, and only if, $T\star S\in \cS'_{\omega}(\R^N)$ for all $S\in \cS'_\omega(\R^N)$. Moreover, if $T\in\cO'_{C,\omega}(\R^N)$, then the linear operators $C_T: \mathcal{S}_{\omega}(\mathbb{R}^N)\to\mathcal{S}_{\omega}(\mathbb{R}^N)$ defined by $C_T(f):=T\star f$, for  $f\in\mathcal{S}_{\omega}(\mathbb{R}^N)$, and $C_T\colon \mathcal{S}'_\omega(\mathbb{R}^N)\to  \mathcal{S}'_\omega(\mathbb{R}^N)$ defined by $C_T(S):=T\star S$, for $S\in \mathcal{S}'_\omega(\mathbb{R}^N)$, are continuous.

We also recall that   
the  inclusions $\mathcal{D}_{\omega}(\mathbb{R}^N) \hookrightarrow \mathcal{S}_{\omega}(\mathbb{R}^N) \hookrightarrow \mathcal{O}_ {C,\omega}(\mathbb{R}^N)\hookrightarrow \mathcal{O}_ {M,\omega}(\mathbb{R}^N)\hookrightarrow \mathcal{E}_ {\omega}(\mathbb{R}^N)$ are well-defined,   continuous with dense range, see  \cite[Theorems 3.8, 3.9 and 5.2(1)]{AC}.
Hence, the inclusions $\mathcal{E}'_{\omega}(\mathbb{R}^N) \hookrightarrow \mathcal{O}'_ {M,\omega}(\mathbb{R}^N)\hookrightarrow \mathcal{O}'_ {C,\omega}(\mathbb{R}^N) \hookrightarrow  \mathcal{S}'_{\omega}(\mathbb{R}^N)\hookrightarrow \mathcal{D}'_ {\omega}(\mathbb{R}^N)  $
are also well-defined and continuous. On the other hand, $\cO_{M,\omega}(\R^N)\hookrightarrow \cS'_\omega(\R^N)$ and
$\cO_{C,\omega}(\R^N)\hookrightarrow \cS'_\omega(\R^N)$ continuously, as it is easy to see.

In \cite{AC2} the following spaces have been introduced.

	\begin{defn}\label{D.SapziOOP} Let $\omega$ be a non-quasianalytic weight function and $1\leq p<\infty$.

	(a) The space $\cO_{M,\omega,p}(\R^N)$ is defined by $	\cO_{M,\omega,p}(\R^N):=\bigcap_{m=1}^{\infty}\bigcup_{n=1}^\infty \mathcal{O}^m_{n,\omega,p}(\mathbb{R}^N)$,
	where
	\small \begin{align*}
	&\cO_{n,\omega,p}^m(\R^N):=	\\
	&\left\{f\in C^\infty(\R^N )\colon r^p_{m,n,p}(f):= \sum_{\alpha\in\mathbb{N}^N_0} \|\exp(-n\omega)\partial^\alpha f\|_p^p\exp\left(-mp\varphi^*_\omega\left(\frac{|\alpha|}{m}\right)\right)<\infty\right\}
	\end{align*}
endowed with the norm $r_{m,n,p}$ is  Banach space for any $m,n\in\N$. The space
	is endowed with its natural lc-topology, i.e., $\cO_{M,\omega,p}(\R^N) =\proj_{\stackrel{\leftarrow}{m}}\,\ind_{\stackrel{\rightarrow}{n}}\,\mathcal{O}^m_{n,\omega,p}(\mathbb{R}^N)$ is the projective limit of the  (LB)-spaces $\cO^m_{\omega,p}(\R^N):=\ind_{\stackrel{\rightarrow}{n}}\,\mathcal{O}^m_{n,\omega,p}(\mathbb{R}^N)$.

	(b) The space  $\cO_{C,\omega,p}(\R^N)$ is defined  by $	\cO_{C,\omega,p}(\R^N):=\bigcup_{n=1}^{\infty}\bigcap_{m=1}^\infty \mathcal{O}^m_{n,\omega,p}(\mathbb{R}^N)$. $\cO_{C,\omega,p}(\R^N)$ endowed with its natural lc-topology  is an (LF)-space, i.e.,  $\cO_{C,\omega,p}(\R^N)=\ind_{\stackrel{\rightarrow}{n}}\cO_{n,\omega,p}(\R^N)$, where $\cO_{n,\omega,p}(\R^N):=\proj_{\stackrel{\leftarrow}{m}}\cO_{n,\omega,p}^m(\R^N)$.
\end{defn}

As shown in  \cite[Proposition 3.8]{AC2}  it holds true  that  $\cO^m_\omega(\R^N)=\cO^m_{\omega,p}(\R^N)$ and $\cO_{n,\omega}(\R^N)=\cO_{n,\omega,p}(\R^N)$ algebraically and topologically for any $1\leq p<\infty$, thereby obtaining that 
 $\cO_{M,\omega}(\R^N)=\cO_{M,\omega,p}(\R^N)$ and $\cO_{C,\omega}(\R^N)=\cO_{C,\omega,p}(\R^N)$ algebraically and topologically for any $1\leq p< \infty$. Since the Banach spaces $\cO_{n,\omega,p}^m(\R^N)$ are reflexive whenever $1<p<\infty$, by a result of Vogt \cite[Proposition 4.4]{Vo} we can conclude that  $\cO^m_\omega(\R^N)=\cO^m_{\omega,p}(\R^N)$  is a reflexive and complete (LB)-space, hence a \textit{regular (LB)-space}, %($\cO_{C,\omega}(\R^N)=\ind_{\stackrel{\rightarrow}{n}}\cO_{n,\omega,p}(\R^N)$ is a \textit{regular (LF)-space}, resp.), 
 being it an inductive limit of reflexive Banach spaces. Recall that an (LF)-space $E=\ind_{\stackrel{\rightarrow}{n}}
 E_n$ is said to be \textit{regular} if every bounded subset in $E$ is
 contained and bounded in $E_n$ for some $n \in\N$ and that complete (LF)-spaces are always regular.
  %{\bf To show that the space  $\cO_{C,\omega}(\R^N)$ is also complete and hence regular, we first observe that proceeding as in \cite[Proposition 3.5]{De2} (take $v_n(x)=\exp(-n\omega)(x)$ for $n\in\N$ and $x\in\R^N$) we can also represent $\cO_{C,\omega}(\R^N)$  as an inductive limit of Fr\'echet Schwartz spaces. Therefore, since its strong dual $\cO'_{C,\omega}(\R^N)$ is isomorphic to $\cO_{M,\omega}(\R^N)$ via the Fourier transform  (\cite[Theorem 6.1]{AC2}) under the assumption $\log(1+t)=o(\omega(t))$ as $t\to\infty$, and hence ultrabornological, we get in such a case thanks to \cite[Theorem 3.5]{We} that the space $\cO_{C,\omega}(\R^N)$ satisfies the condition $(wQ)$, i.e., 
 %	\begin{align*}
 %	& \forall n\in\N \ \exists m>n,\ N\in\N \ \forall k>m,\ M\in\N\ \exists K\in\N,\ S>0\ \forall f\in \cO_{n,\omega}(\R^N)\\&\quad\quad \qquad
% 	r_{M,m}(f)\leq S(r_{K,k}(f)+r_{N,n}(f))
% 	\end{align*}
% 	(we remark that (LB)-spaces always satisfy the condition $(wQ)$).
% Under the assumption $\log(1+t)=o(\omega(t))$ as $t\to\infty$, since $\cO_{C,\omega}(\R^N)$ is an inductive limit of reflexive Fr\'echet spaces and satisfies the condition $(wQ)$, we can apply \cite[Proposition 3.5]{Vo} to conclude that  $\cO_{C,\omega}(\R^N)$ is a reflexive and complete (LF)-space, hence a regular (LF)-space.}
 An (LF)-space $E=\ind_{\stackrel{n}{\to}}E_n$ is said to be  \textit{sequentially retractive} if for every null sequence in $E$ there exists
 $n\in\N$ such that the sequence is contained and converges to zero in $E_n$. We observe that sequential retractivity implies completeness, see	\cite[Corollary 2.8]{We}.
 
 We end this section by showing that the (LB)-spaces $\cO^m_\omega(\R^N)$, with $m\in\N$, 
 %and $\cO_{C,\omega}(\R^N)$, 
  are sequentially retractive and Montel. For this, we need the following lemma.
 
\begin{lem}\label{Lbou} Let $\omega$ be a non-quasianalytic weight function and let $m\in\N$. Then the spaces $\cO_{n+1,\omega}^m(\R^N)$ and $\cE_\omega(\R^N)$ induce the same topology on the closed unit ball $B_{n}^m$ of  $\cO_{n,\omega}^m(\R^N)$ for all $n\in\N$.
\end{lem}	

\begin{proof} Fix $n\in\N$ and $\varepsilon>0$. The set $U:=\{f\in \cO_{n+1,\omega}^m(\R^N)\colon r_{m,n+1}(f)<\varepsilon\}$ is a $0$-neighborhood of $\cO_{n+1,\omega}^m(\R^N)$. Now, let $M>0$ such that $\exp(-\omega(x))<\varepsilon$ for every $|x|\geq M$ and $V:=\{f\in \cE_\omega(\R^N)\colon p_{K,m}(f)<\varepsilon \}$, where $K:=\{x\in\R^N\colon |x|\leq M\}$. Then $V\cap B_{n}^m\su U\cap B_{n}^m$. Indeed,   if $f\in V\cap B_{n}^m$, then 
	\[
	 |\partial^\alpha f(x)|\leq \exp\left(n\omega(x)+m\varphi^*_\omega\left(\frac{|\alpha|}{m}\right)\right)< \varepsilon \exp\left((n+1)\omega(x)+m\varphi^*_\omega\left(\frac{|\alpha|}{m}\right)\right)
	\]
	for every $\alpha\in\N_0^N$ and $|x|\geq M$. Moreover,  $f\in V\cap B_{n}^m$ also  implies that 
	\[
	|\partial^\alpha f(x)|<\varepsilon \exp\left(m\varphi^*_\omega\left(\frac{|\alpha|}{m}\right)\right)\leq \varepsilon\exp\left((n+1)\omega(x)+m\varphi^*_\omega\left(\frac{|\alpha|}{m}\right)\right)
	\]
	for every $\alpha\in\N_0^N$ and $|x|\leq M$. It follows that  $r_{m,n+1}(f)<\varepsilon$ and so, $f\in U\cap B_n^m$. Since $\varepsilon>0$ is arbitrary, we get the thesis.
%	
%	(ii) We can assume that $B$ is absolutely convex. Hence, we have only to show that the topology of $\cE_\omega(\R^N)$ induces a finer filter of $0$-neighborhoods in $B$ than the topology of $\cO_{n+1,\omega}(\R^N)$.  So, let $m\in\N$ and $\varepsilon>0$.  Then the set $U:=\{f\in \cO_{n+1,\omega}(\R^N)\colon r_{m,n+1}(f)<\varepsilon\}$ is a $0$-neighborhood of  $\cO_{n+1,\omega}(\R^N)$. Since $B$ is a bounded subset of  $\cO_{n,\omega}(\R^N)$, there exists $C_m>0$ such that $r_{m,n}(f)\leq C_m$ for all $f\in B$. Let $M>0$ such that $\exp(-\omega(x))<\frac{\varepsilon}{C_m}$ for every $|x|\geq M$ and $V:=\{f\in \cE_\omega(\R^N)\colon p_{K,m}(f)<\varepsilon\}$, where $K:=\{x\in\R^N\colon |x|\leq M\}$. Prooceding in an analogous way as in point (i), one shows that $V\cap B\su U\cap B$. %Indeed, if $f\in V\cap B$, then
%	\[
%	|\partial^\alpha f(x)|\leq C_m\exp\left(n\omega(x)+m\varphi^*_\omega\left(\frac{|\alpha|}{m}\right)\right)< \varepsilon \exp\left((n+1)\omega(x)+m\varphi^*_\omega\left(\frac{|\alpha|}{m}\right)\right)
%	\]
%	for every $\alpha\in\N_0^N$ and $|x|\geq M$. Moreover
%	\[
%	|\partial^\alpha f(x)|<\varepsilon \exp\left(m\varphi^*_\omega\left(\frac{|\alpha|}{m}\right)\right)\leq \varepsilon\exp\left((n+1)\omega(x)+m\varphi^*_\omega\left(\frac{|\alpha|}{m}\right)\right)
%	\]
%	for every $\alpha\in\N_0^N$ and $|x|\leq M$. It follows that $r_{m,n+1}(f)<\varepsilon$ and hence, $f\in U\cap B$.
	\end{proof}

We can  to  show that the (LB)-spaces $\cO^m_\omega(\R^N)$, with $m\in\N$,
% and $\cO_{C,\omega}(\R^N)$
 are sequentially retractive.

% The result for the last case should be compared with \cite[Corollary 3.9(ii)]{De2} valid in the setting of ultradifferentiable classes defined according to Komatsu, by pointing out that such a result does not apply to our case.
	
\begin{thm}\label{T.complete} Let $\omega$ be a non-quasianalytic weight function. For every  $m\in\N$ the  (LB)-space $\cO^m_\omega(\R^N)$ is  sequentially retractive.
%	
%	Then the following assertions hold:
%	\begin{itemize}
%		\item[(i)] For $m\in\N$, $\cO^m_\omega(\R^N)$ is a sequentially retractive (LB)-space;
%		\item[(ii)] If $\log(1+t)=o(\omega(t))$ as $t\to\infty$, then $\cO_{C,\omega}(\R^N)$ is a sequentially rectractive (LF)-space.
%	\end{itemize}
%Therefore, the spaces $\cO^m_\omega(\R^N)$, for $m\in\N$, and $\cO_{C,\omega}(\R^N)$ are complete.
	\end{thm}

\begin{proof}  Let $\{f_j\}_{j\in\N}$ be a null sequence  of $\cO^m_\omega(\R^N)$. Then $B:=\{f_j\colon j\in\N\}$ is a bounded subset of $\cO^m_\omega(\R^N)$. Since $\cO^m_\omega(\R^N)$ is a regular (LB)-space, $B$ is contained and bounded in $\cO^m_{n,\omega}(\R^N)$ for some $n\in\N$. On the other hand,  $\{f_j\}_{j\in\N}$ converges to $0$  in $\cE_\omega(\R^N)$, as $\cO^m_\omega(\R^N)$ is continuously included in  $\cE_\omega(\R^N)$. Since $\cO^m_{n+1,\omega}(\R^N)$ and $\cE_\omega(\R^N)$ induce the same topology  on $B$ by Lemma \ref{Lbou}, it follows that $\{f_j\}_{j\in\N}$ converges to $0$  in $\cO_{n+1,\omega}^m(\R^N)$.
%	
%	(ii) The proof follows by arguing as in (i) and applying Lemma \ref{Lbou}(ii).
	\end{proof}

\begin{thm}\label{T.Montelspaces} Let $\omega$ be a non-quasianalytic weight function. Then the spaces $\cO_\omega^m(\R^N)$, for $m\in\N$, and $\cO_{M,\omega}(\R^N)$ are complete Montel spaces. %If $\log(1+t)=o(\omega(t))$ as $t\to\infty$, $\cO_{C,\omega}(\R^N)$ is a complete Montel space.
	\end{thm}

\begin{proof} The spaces $\cO_\omega^m(\R^N)$, for $m\in\N$, are clearly barrelled as inductive limits of barrelled spaces. On the other hand, $\cO_{M,\omega}(\R^N)$ is ultrabornological by \cite{De}, hence barrelled. Hence, it remains to show that each bounded set in such spaces is relativey compact. This follows from Lemma \ref{Lbou}. Indeed, fixed $m\in\N$ and  a bounded subset $B$ of $\cO_\omega^m(\R^N)$, from the regularity of $\cO_\omega^m(\R^N)$ there exists $n\in\N$ such that $B$ is contained and bounded in $\cO^m_{n,\omega}(\R^N)$. By  Lemma \ref{Lbou} the spaces  $\cO^m_{n+1,\omega}(\R^N)$ and $\cE_\omega(\R^N)$ induce the same topology on $B$ and hence, $B$ is relatively compact in  $\cO^m_{n+1,\omega}(\R^N)$, after having observed that $\cE_\omega(\R^N)$ is Montel. Therefore, $B$ is also  relatively compact in  $\cO^m_{\omega}(\R^N)$. 

Since $\cO_{M,\omega}(\R^N)$ is the projective limit of the Montel spaces  $\cO^m_{\omega}(\R^N)$, each bounded subset of  $\cO_{M,\omega}(\R^N)$ is clearly relatively compact. The completeness of the
	space $\cO_{M,\omega}(\R^N)$  follows from the fact that it is the projective limit of the complete spaces $\cO^m_\omega(\R^N)$.
%	
%By arguing in a similar way to the case $\cO^m_\omega(\R^N)$ with the use of Lemma \ref{Lbou}(ii), one shows that each bounded subset of $\cO_{C,\omega}(\R^N)$ is relatively compact. Finally,  the space $\cO_{C,\omega}(\R^N)$ is complete as already observed.
\end{proof}

\begin{rem} We remark that the completeness of the space $\cO_{M,\omega}(\R^N)$ follows also from \cite[Theorem 5.2(ii)]{AC} combined with \cite[Theorem 5.3]{De}.
	\end{rem}

\section{The Algebras $(\cO_{M,\omega}(\R^N),\cdot)$, $(\cS_{\omega}(\R^N),\cdot)$, $(\cS_\omega(\R^N),\star)$ and $(\cO'_{C,\omega}(\R^N),\star)$}
Let us recall that a bilinear mapping $b\colon E \times F\to  G$ between locally convex Hausdorff
spaces $E$, $F$, $G$ is continuous if, and only if, for every continuous seminorm $p_1$
on $G$ there exist continuous seminorms $p_2$ and $p_3$ on $E$ and $F$, respectively, such
that the inequality
\[
p_1(b(v,w)) \leq p_2(v)p_3(w)
\]
holds for every pair $(v,w) \in E\times  F$.
A locally convex algebra (topological algebra, briefly) over the field $\mathbb{K}=\R$ or $\mathbb{K}=\C$ is  a lcHs $E$ together with a bilinear map $b\colon E\times E\to E$ that turns $E$ into an algebra over $\mathbb{K}$ with $b$  continuous.

The aim of this section is to establish  that $(\cO_{M,\omega}(\R^N),\cdot)$, $(\cS_{\omega}(\R^N),\cdot)$ are multiplication topological algebras, while $(\cS_\omega(\R^N),\star)$ and $(\cO'_{C,\omega}(\R^N),\star)$ are convolution topological algebras.

We first consider  the case $\cO_{M,\omega}(\R^N)$ and establish the following fact.

\begin{lem}\label{L.FQ} Let  $\omega$ be a 
	non-quasianalytic weight function. Then for all $k\in \cS_\omega(\R^N)$ there exists $l\in \cS_\omega(\R^N)$  such that $|k(x)|\leq l^2(x)$ for every $x\in\R^N$.
\end{lem}

\begin{proof} Let $k\in \cS_\omega(\R^N)$ be fixed. The function $h(x):=\sqrt{|k(x)|}$ for $x\in\R^N$ is a  non-negative function such that 
	\[
	\lim_{|x|\to\infty}\exp(\lambda\omega(x))h(x)=\lim_{|x|\to\infty}\sqrt{\exp\left(\frac{\lambda}{2}\omega(x)\right)|k(x)|}=0
	\]
	for all $\lambda>0$. By \cite[Lemma 5.8]{AC} there exists $l\in\cS_\omega(\R^N)$ for which $h(x)\leq l(x)$ for every $x\in\R^N$. Accordingly, $|k(x)|\leq l^2(x)$ for every $x\in\R^N$.
\end{proof}

%	\begin{rem}\label{R.Altre} Let $\omega$ be a 
%		non-quasianalytic weight function and $k_1,\ldots, k_n\in \cS_\omega(\R^N)$. Then there exists $g\in \cS_\omega(\R^N)$ such that $\sup_{i=1,\ldots,n}|k_i(x)|\leq g(x)$ for every $x\in\R^N$. The proof follows after having observed that 
%	\[
%	\lim_{|x|\to\infty}\exp(\lambda\omega(x))\sup_{i=1,\ldots,n}|k_i(x)|=0
%	\]	
%	for every $\lambda>0$, and  then applying \cite[Lemma 5.8]{AC}. 
%	\end{rem}

\begin{thm}\label{mu} Let $\omega$ be a non-quasianalytic weight function. Then  $(\cO_{M,\omega}(\R^N),\cdot)$  is a multiplication topological algebra.
\end{thm}

\begin{proof}
	To see this, we fix	$m\in\N$  and $k\in \cS_\omega(\R^N)$. We choose $m'\in\N$ such that $m'\geq Lm$, where $L$ is the constant appearing in \eqref{l}, and $l\in\cS_\omega(\R^N)$ as in Lemma \ref{L.FQ}. Now, let $f,g \in \mathcal{O}_{M,\omega}(\mathbb{R}^N)$. Then we have for every $\alpha \in \mathbb{N}^N_0$ and $x\in \mathbb{R}^N$ that 
	\begin{align*}
	|k(x)||\partial^\alpha (fg)(x)|&%\leq l^2(x) \sum_{\gamma\leq\alpha}\binom{\alpha}{\gamma} |\partial^\gamma f(x)||\partial^{\alpha-\gamma} g(x)|\\
	\leq \sum_{\gamma\leq\alpha}\binom{\alpha}{\gamma} l(x)|\partial^\gamma f(x)|l(x)|\partial^{\alpha-\gamma} g(x)|\\&\leq 
	q_{m',l}(f)q_{m',l}(g) \sum_{\gamma\leq\alpha}\binom{\alpha}{\gamma} \exp\left(m'\varphi^*_\omega\left(\frac{|\gamma|}{m'}\right)\right)\exp\left(m'\varphi^*_\omega\left(\frac{|\alpha-\gamma|}{m'}\right)\right).\\
	%	&=	q_{m',l}(f)q_{m',l}(g)  \sum_{\gamma\leq\alpha}\binom{\alpha}{\gamma} \exp\left(m'\varphi^*_\omega\left(\frac{|\gamma|}{m'}\right)+m'\varphi^*_\omega\left(\frac{|\alpha-\gamma|}{m'}\right)\right).
	\end{align*}
	Applying \eqref{secondprop} and \eqref{firstprop}, we obtain for every $\alpha\in\N_0^N$ and $x\in\R^N$ that 
	\begin{align*}
	|k(x)||\partial^\alpha (fg)(x)|&%\leq 	q_{m',l}(f)q_{m',l}(g) \sum_{\gamma\leq\alpha}\binom{\alpha}{\gamma} \exp\left(m'\varphi^*_\omega\left(\frac{|\alpha|}{m'}\right)\right)\\
	\leq 	q_{m',l}(f)q_{m',l}(g)2^{|\alpha|}\exp\left(m'\varphi^*_\omega\left(\frac{|\alpha|}{m'}\right)\right)\\
	&\leq  C	q_{m',l}(f)q_{m',l}(g)\exp\left(m\varphi^*_\omega\left(\frac{|\alpha|}{m}\right)\right), 
	\end{align*}
	where 	$C:=e^{mL}$. Accordingly,
	\begin{equation}\label{e.mult}
	q_{m,k}(fg)\leq C 	q_{m',l}(f)q_{m',l}(g).
	\end{equation}
	Since $f,g\in \cO_{M,\omega}(\R^N)$, $m\in\N$ and $k\in\cS_\omega(\R^N)$ are arbitrary, we can conclude from \eqref{e.mult}  that   the multiplication operator $M\colon \cO_{M,\omega}(\R^N)\times \cO_{M,\omega}(\R^N)\to \cO_{M,\omega}(\R^N)$, $(f,g)\mapsto fg$, is well-defined and  continuous, i.e., $\cO_{M,\omega}(\R^N)$ is a multiplication topological algebra. 
\end{proof}

Arguing in a similar way with simple changes (indeed, it suffices to consider $\exp(n\omega)$ instead of the function $k$ and to take $\exp\left(\frac{n}{2}\omega\right)$ instead of the function $l$), one shows the same result for $\cS_{\omega}(\R^N)$.

\begin{thm}\label{muS} Let $\omega$ be a non-quasianalytic weight function. Then  $(\cS_{\omega}(\R^N),\cdot)$  is a multiplication topological algebra.
\end{thm}

We now introduce the following definition.

\begin{defn} Let $\omega$ be a non-quasianalytic weight function.
	Let $E$ be a lcHs of $\omega$-ultradifferentiable functions on $\mathbb{R}^N$ continuously included in $\cE_\omega(\R^N)$ with dense range. We denote by  $M(E)$ the space of all multipliers of $E$, i.e., the largest space of $\omega$-ultradifferentiable functions on $\R^N$ satisfying  the following conditions:
	\begin{enumerate}
		\item  the multiplication operator on $E\times M(E) \to \cE_\omega(\R^N)$, $(f,g)\mapsto fg$, is well-defined and takes values in $E$;
		\item for all $f\in M(E)$, the operator $M_f\colon E\to E$, $g\mapsto  fg$ is continuous.
	\end{enumerate}
	If $E'$ is the strong dual of $E$, we denote by $M(E')$  the space of all multipliers of $E'$, i.e., the largest space of $\omega$-ultradistributions on $\R^N$ for which the following conditions are satisfied:
	\begin{enumerate}
		\item for all $T\in E'$ and $f\in M(E')$ we have that $fT$ is well-defined on $E$ and belongs to $E'$;
		\item for all $f\in M(E')$, the operator $M_f\colon E'\to E'$, $T\mapsto  fT$ is continuous.
	\end{enumerate}
\end{defn}

\begin{rem}\label{ossm} We first recall that $M(\cS_\omega(\R^N))=M(\cS'_\omega(\R^N))=\cO_{M,\omega}(\R^N)$.

	We now observe that if $E$ is a  lcHs of $\omega$-ultradifferentiable functions on $\mathbb{R}^N$ continuously included in $\cE_\omega(\R^N)$ with dense range and such that the constant functions belong to $E$, then
	\begin{equation*}
	M(E) = \{1\}\cdot M(E) \su E \cdot M(E) \su E.
	\end{equation*}
	Since $\mathcal{O}_{M,\omega}(\mathbb{R}^N)$ and $ \mathcal{O}_{C,\omega}(\mathbb{R}^N)$  contain the constant functions, we get 
	$M(\mathcal{O}_{M,\omega}(\mathbb{R}^N))\su \mathcal{O}_{M,\omega}(\mathbb{R}^N)$ and $M(\mathcal{O}_{C,\omega}(\mathbb{R}^N))\su \mathcal{O}_{C,\omega}(\mathbb{R}^N)$. But $\cO_{M,\omega}(\R^N)$ is a multiplication algebra as shown in Theorem \ref{mu}. Therefore, we clearly have that $\cO_{M,\omega}(\R^N)\su M(\mathcal{O}_{M,\omega}(\mathbb{R}^N))$. Thus, $M(\cO_{M,\omega}(\R^N))=\cO_{M,\omega}(\R^N)$.
\end{rem}	

It is also true that	
$M(\cO_{C,\omega}(\R^N))=\cO_{C,\omega}(\R^N)$, as the following result shows.

\begin{prop}\label{P.O_C-M} Let $\omega$ be a non-quasianalytic weight function. Then $M(\cO_{C,\omega}(\R^N))=\cO_{C,\omega}(\R^N)$. Hence, $(\cO_{C,\omega}(\R^N),\cdot)$ is a multiplication algebra.
\end{prop}

\begin{proof} By Remark \ref{ossm} it suffices to show that 
	$\cO_{C,\omega}(\R^N)\su M(\cO_{C,\omega}(\R^N))$. In order to do this, we fix $n,m\in\N$ with $n\geq 2$ and choose $m'\geq Lm$, with $L$ the constant appearing in \eqref{l}. Let $n_1,n_2\in\N$ be such that $n_1+n_2=n$. If $f\in \cO_{n_1,\omega}(\R^N)=\cap_{r=1}^\infty \cO_{n_1,\omega}^{r}(\R^N)$ and $g\in\cO_{n_2,\omega}(\R^N)= \cap_{r=1}^\infty \cO_{n_2,\omega}^{r}(\R^N)$, then via \eqref{secondprop} and \eqref{firstprop} we have for every $\alpha\in\N_0^N$ and $x\in\R^N$ that 
	\begin{align*}
	|\partial^\alpha(fg)(x)|&\leq \sum_{\gamma\leq \alpha }\binom{\alpha}{\gamma}|\partial^\gamma f(x)||\partial^{\alpha-\gamma}g(x)|\\&\leq
	r_{m',n_1}(f)r_{m',n_2}(g)\sum_{\gamma\leq \alpha }\binom{\alpha}{\gamma}\exp\left(n_1\omega(x)+m'\varphi_\omega^*\left(\frac{|\gamma|}{m'}\right)\right)\times\\
	&\qquad \times \exp\left(n_2\omega(x)+m'\varphi_\omega^*\left(\frac{|\alpha-\gamma|}{m'}\right)\right)\\
	%	&= 	r_{m',n_1}(f)r_{m',n_2}(g)\exp(n\omega(x))\sum_{\gamma\leq \alpha }\binom{\alpha}{\gamma}\exp\left(m'\varphi_\omega^*\left(\frac{|\gamma|}{m'}\right)+m'\varphi_\omega^*\left(\frac{|\alpha-\gamma|}{m'}\right)\right)\\
	%	&\leq r_{m',n_1}(f)r_{m',n_2}(g)\exp(n\omega(x))2^{|\alpha|}\exp\left(m'\varphi_\omega^*\left(\frac{|\alpha|}{m'}\right)\right)\\
	&\leq C r_{m',n_1}(f)r_{m',n_2}(g)\exp(n\omega(x))\exp\left(m\varphi_\omega^*\left(\frac{|\alpha|}{m}\right)\right),
	\end{align*}
	where $C:=e^{mL}$.
	Accordingly, we have
	\begin{equation}\label{e.Multip}
	r_{m,n}(fg)\leq C r_{m',n_1}(f)r_{m',n_2}(g).
	\end{equation}
	Since $m\in\N$, $f\in\cO_{n_1,\omega}(\R^N) $ and $g\in  \cO_{n_2,\omega}(\R^N)$ are arbitrary, from \eqref{e.Multip} it follows that the multiplication operator $M\colon  \cO_{n_1,\omega}(\R^N)\times  \cO_{n_2,\omega}(\R^N)\to  \cO_{n,\omega}(\R^N)$ is continuous. But also $n\in\N$ is arbitrary and $\cO_{C,\omega}(\R^N)$ is the inductive limit of the Fr\'echet spaces $ \cO_{n,\omega}(\R^N)$. So, we can conclude that the multiplication operator $M\colon \cO_{C,\omega}(\R^N)\times \cO_{C,\omega}(\R^N)\to \cO_{C,\omega}(\R^N)$ is well-defined. Moreover,  \eqref{e.Multip} also implies that the operator $M_f\colon \cO_{C,\omega}(\R^N)\to \cO_{C,\omega}(\R^N)$, $g\mapsto fg$, is  continuous for all $f\in \cO_{C,\omega}(\R^N)$. Therefore, $\cO_{C,\omega}(\R^N)\su M(\cO_{C,\omega}(\R^N))$.
	This completes the proof.
\end{proof}
\begin{rem}
	For any non-quasianalytic weight function $\omega$ such that  $\log(1+t)=o(\omega(t))$ as $t\to\infty$., the space $(\cO_{C,\omega}(\R^N),\cdot)$ is not a multiplication topological algebra (see Section 4).  
\end{rem}
\begin{cor}\label{C.Moltiplicatori} Let $\omega$ be a non-quasianalytic weight function. Then $M(\mathcal{O}'_{M,\omega}(\mathbb{R}^N))= \mathcal{O}_{M,\omega}(\mathbb{R}^N)$ and $ M(\mathcal{O}'_{C,\omega}(\mathbb{R}^N))= \mathcal{O}_{C,\omega}(\mathbb{R}^N)$. 
\end{cor}

\begin{proof} The result follows from  Theorem \ref{mu}  and Proposition \ref{P.O_C-M}, taking into account of the fact that 
	the multiplication of $\omega$-ultradistributions  with $\omega$-ultradifferentiable functions is defined by transposition.
	\end{proof}

We now pass to the case $(\cS_\omega(\R^N),\star)$, for which the following result is true.

\begin{thm}\label{bili}
	Let $\omega$ be a non-quasianalytic weight function. Then $(\cS_\omega(\R^N),\star)$ is a convolution topological algebra.
\end{thm}
\begin{proof}
	To see this we have to show that the bilinear map $\star\colon  \cS_\omega(\R^N)\times \cS_\omega(\R^N)\to \cS_\omega(\R^N)$ is continuous.
	So, let $n_0\in\N$ be fixed such that $n_0\geq \frac{N+1}{b}$. Hence, $\exp(-n_0\omega)\in L^1(\R^N)$ by \eqref{eq.Lpspazi}. Fixed $n\in\N$ and $f,g \in \cS_\omega(\R^N)$, we have
	\begin{align*}
	\partial^\alpha (f\star g)(x)=\int_{\mathbb{R}^N} f(y)\partial^\alpha g (x-y)\, dy, \; x\in\R^N, \; \alpha\in\N^N_0.
	\end{align*}
	By \eqref{sub} $\omega(x)=\omega(x-y+y)\leq K(1+\omega(x-y)+\omega(y))$. Hence, we get for every $x\in \R^N$ and $\alpha \in \N^N_0$ that
	\begin{align*}
	|\partial^\alpha (f\star g)(x)|\exp(n\omega(x))& \leq e^{Kn}\int_{\mathbb{R}^N} |f(y)|\exp(Kn\omega(y))|\partial^\alpha g(x-y)|\exp(Kn\omega(x-y))\, dy\\ &\leq e^{Kn}\|\exp(-n_0\omega)\|_1 \|f\exp((Kn+n_0)\omega)\|_\infty\|\partial^\alpha g\exp(Kn\omega)\|_\infty.
	\end{align*}
	%	Accordingly,  we obtain for every $m\in\N$ and $\alpha\in \N_0^N$ that
	%	\begin{align*}
	%	&	\|\partial^\alpha (\phi\star \psi)\exp(n\omega)\|_\infty \exp\left(-m\varphi^*\left(\frac{|\alpha|}{m}\right)\right)\leq\\ &\, \leq e^{Kn}\|\exp(-n_0\omega)\|_1\|\phi\exp((nK+n_0)\omega)\|_\infty\|\partial^\alpha\psi\exp(nK\omega)\|_\infty \exp\left(-m\varphi^*\left(\frac{|\alpha|}{m}\right)\right).
	%	\end{align*}
	Therefore, for all $m,n\in\N$
	\begin{align*}
	q_{m,n}(f\star g)&\leq e^{Kn}\|\exp(-n_0\omega)\|_1 \|f\exp((Kn+n_0)\omega)\|_\infty q_{m,Kn}(g)\\ 
	&\leq e^{Kn}\|\exp(-n_0\omega)\|_1 q_{m,K n+n_0}(f)q_{m,Kn}(g)\\
	&\leq e^{Kn}\|\exp(-n_0\omega)\|_1 q_{m,n'}(f)q_{m,n'}(g), 
	\end{align*}
	with  $n'\in\N$ such that $n'\geq Kn+n_0$.
\end{proof}
\begin{rem} We point out that the result in Theorem \ref{bili} can be achieved by applying Theorem \ref{muS} combined with the use of the Fourier transform, which is a topological isomorphism from $\cS_\omega(\R^N)$ onto itself such that $\widehat{f\star g}=\hat{f}\hat{g}$, for all $f,g\in  \cS_\omega(\R^N)$, see \cite{B}.
	
\end{rem}

The next aim is to show  that  $(\cO'_{C,\omega}(\R^N),\star)$ is a convolution topological  algebra. Hence, let us recall that 
in case $T\in \cO'_{C,\omega}(\R^N)$ and $S\in \cS'_\omega(\R^N)$ the convolution $T\star S$ is well defined on $\cS_\omega(\R^N)$ and belongs  to $\cS'_{\omega}(\R^N)$. Indeed,  for all $f\in \cS_\omega(\R^N)$ we have $\check{T}\star f\in \cS_\omega(\R^N)$ and the operator $C_T\colon \cS_\omega(\R^N)\to  \cS_\omega(\R^N)$, $f\mapsto \check{T}\star f$, is continuous, see, \cite[Theorem 5.3]{AC2} (the distribution $\check{T}$ is  defined by  $f\mapsto \check{T}(f):=T(\check{f})$, with $\check{f}(x):=f(-x)$ for all $x\in\R^N$. Furthermore, for $R$ a distribution and $f$ a function the convolution $R\star f$ is defined by $(R\star f)(x):=\langle R_y,\tau_x\check{f}\rangle$, where the notation $R_y$ means that the distribution $R$ acts on a function $\phi(x -y)$,
when the latter is regarded as a function of the variable $y$). Hence, the  convolution $T\star S$ defined by $\langle T\star S, f\rangle =\langle S, \check{T}\star f\rangle$, for $f\in \cS_\omega(\R^N)$, is clearly a well-defined element of $\cS'_\omega(\R^N)$. Since $C_T$ is a continuous operator from $\cS_\omega(\R^N)$ into $\cO_{C,\omega}(\R^N)$ and the space $\cS_\omega(\R^N)$ is dense in $\cO_{C,\omega}(\R^N)$, we also have that $T\star S\in \cO'_{C,\omega}(\R^N)$ (in the sense that $T\star S$ extends continuously on whole  $\cO_{C,\omega}(\R^N)$).

Moreover, for any $\omega$ non-quasianalytic weight function satisfying the condition $\log(1+t)=o(\omega(t))$ as $t\to\infty$, the  Fourier transform $\cF$ is a topological isomorphism from  the space $\cO'_{C,\omega}(\R^N)$  onto the space $\cO_{M,\omega}(\R^N)$ (see \cite[Theorem 6.1]{AC2}). Accordingly, we have $\mathcal{F}(\mathcal{O}'_{C,\omega}(\mathbb{R}^N))=\mathcal{O}_{M,\omega}(\mathbb{R}^N)$ and $\mathcal{F}(\mathcal{O}_{M,\omega}(\mathbb{R}^N))=\mathcal{O}'_{C,\omega}(\mathbb{R}^N$. In particular, for all $T\in \cO'_{C,\omega}(\R^N)$ and $S\in \cS'_\omega(\R^N)$  the convolution $T\star S$ satisfies the following property:
\begin{equation}\label{eq.FT-C}
\cF(T\star S)=\cF(T)\cF(S).
\end{equation}
%(see \cite[Theorem 6.1]{AC2}). %On the other hand, as shown in  \cite[Theorem 6.1]{AC2}  for every $f\in \cO_{M.\omega}(\R^N)$ and $T\in \cS'_\omega(\R^N)$ we following property is also true:
%\begin{equation*}\label{eq.FT-P}
%\cF(fT)=(2\pi)^{-N}\hat{f}\cF(T).
%\end{equation*} 
We can now state the following result.

\begin{thm}\label{T.Convolution} Let $\omega$ be a non-quasianalytic weight function such that $\log(1+t)=o(\omega(t))$ for $t\to\infty$. Then  $(\cO'_{C,\omega}(\R^N),\star)$ is a convolution topological algebra.
\end{thm}

\begin{proof} 
	We first observe that if $S, T\in\cO'_{C,\omega}(\R^N)\subset \cS'_{\omega}(\R^N)$, then the convolutions $T\star S$ and $S\star T$ are well defined and belong to $\cS'_\omega(\R^N)$. Since $\cF(S), \cF(T)\in \cO_{M,\omega}(\R^N)$ and so, $\cF(S)\cF(T)=\cF(T)\cF(S)$,  by \eqref{eq.FT-C} it follows that $S\star T=T\star S\in \cO'_{C,\omega}(\R^N)$. This means that $(\cO'_{C,\omega}(\R^N),\star)$ is a convolution  algebra. Finally,
	since  the Fourier transform is a topological isomorphim from $\cO'_{C,\omega}(\R^N)$ onto $\cO_{M,\omega}(\R^N)$ satisfying equation \eqref{eq.FT-C}, we can apply Theorem \ref{mu} to obtain that $(\cO'_{C,\omega}(\R^N),\star)$ is a topological algebra.  
\end{proof}

As done in the case of multipliers of lcHs' of $\omega$-ultradifferentiable functions on $\R^N$, we now introduce the space of convolutors.

\begin{defn} Let
	$E$ be a lcHs of $\omega$-ultradifferentiable functions on $\mathbb{R}^N$. We denote by $C(E)$ the space of all convolutors of $E$, i.e., the largest space of $\omega$-ultradistributions on $\mathbb{R}^N$ satisfying the following conditions:
	\begin{enumerate}
		\item 
		the convolution operator on $E\times C(E)\to \cE_\omega(\R)$, $(f, T)\mapsto T\star f$, is well-defined  and takes value in $E$;
		\item for all $T\in C(E)$ the operator $C_T\colon E\to E$, $f\mapsto T\star f$, is continuous.
	\end{enumerate} 
	We denote by $C(E')$ the space of all convolutors of $E'$, i.e., the largest space  of $\omega$-ultradistributions on $\mathbb{R}^N$ satisfying the following conditions:
	\begin{enumerate}
		\item for all $S\in E'$ and $T\in C(E')$ we have that $T\star S$ is well-defined on $E$ and belongs to $E'$;
		\item for all $T\in C(E')$ the operator $C_T\colon E'\to E'$, $S\mapsto T\star S$, is continuous.
	\end{enumerate}
\end{defn}

\begin{rem}\label{OssC} Let $\omega$ be a  non-quasianalytic weight function $\omega$ such that $\log(1+t)=o(\omega(t))$ for $t\to\infty$.	Then $C(\cS_\omega(\R^N))=C(\cS'_\omega(\R^N))=\cO'_{C,\omega}(\R^N)$.
	
	We point out that the equation \eqref{eq.FT-C} is also satisfied for any pairs $(T,S)$ of $\omega$-ultradistributions in $\cO'_{C,\omega}(\R^N)\times \cO_{C,\omega}(\R^N)$,  $\cO'_{M,\omega}(\R^N)\times \cO_{M,\omega}(\R^N)$, $\cO'_{C,\omega}(\R^N)\times \cO'_{C,\omega}(\R^N)$ and $\cO'_{M,\omega}(\R^N)\times \cO'_{M,\omega}(\R^N)$, because these spaces are continuously included in $\cO'_{C,\omega}(\R^N)\times \cS'_{\omega}(\R^N)$. On the other hand, by Remark \ref{ossm} and Proposition \ref{P.O_C-M} we have that $M(\cO_{C,\omega}(\R^N))=\cO_{C,\omega}(\R^N)$ and $M(\cO_{M,\omega}(\R^N))=\cO_{M,\omega}(\R^N)$. Recalling that the Fourier transfom is a topological isomorphism from $\cO'_{C,\omega}(\R^N)$ onto $\cO_{M,\omega}(\R^N)$, these facts yield that $C(\cO'_{M,\omega}(\R^N))
	= \cO'_{M,\omega}(\R^N)$ and $C(\cO'_{C,\omega}(\R^N))= \cO'_{C,\omega}(\R^N)$. Furthemore, by Corollary \ref{C.Moltiplicatori} we have that $M(\cO'_{C,\omega}(\R^N))=\cO_{C,\omega}(\R^N)$ and $M(\cO'_{M,\omega}(\R^N))=\cO_{M,\omega}(\R^N)$. Thus, by the same arguments we obtain  that 
	$C(\cO_{M,\omega}(\R^N))=\cO'_{M,\omega}(\R^N)$ and $C(\cO_{C,\omega}(\R^N))=\cO'_{C,\omega}(\R^N)$.
	%%%%%%%%%%%%%%%%%%%%%

	We summarize our results in this simple table.
	\\[0.5cm]
	\begin{tabular}{|l|l|l|}
		\hline
		E&M(E)&C(E)\\
		\hline
		$\cS_{\omega}(\R^N)$&$\cO_{M,\omega}(\R^N)$&$\cO'_{C,\omega }(\R^N)$\\
		\hline
		$\cS'_{\omega}(\R^N)$&$\cO_{M,\omega}(\R^N)$&$\cO'_{C,\omega }(\R^N)$\\
		\hline
		$\cO_{M,\omega}(\R^N)$&$\cO_{M,\omega}(\R^N)$&$\cO'_{M,\omega}(\R^N)$\\
		\hline
		$\cO'_{M,\omega}(\R^N)$&$\cO_{M,\omega}(\R^N)$&$\cO'_{M,\omega}(\R^N)$\\
		\hline
		$\cO_{C,\omega}(\R^N)$&$\cO_{C,\omega}(\R^N)$&$\cO'_{C,\omega}(\R^N)$\\
		\hline
		$\cO'_{C,\omega}(\R^N)$&$\cO_{C,\omega}(\R^N)$&$\cO'_{C,\omega}(\R^N)$\\
		\hline
	\end{tabular}
\end{rem}

From Remark \ref{OssC}, it follows this result.

\begin{cor} Let $\omega$ be a non-quasianalytic weight function such that $\log(1+t)=o(\omega(t))$ for $t\to\infty$. Then $\cO'_{M,\omega}(\R^N)$ is a convolution algebra.
\end{cor}

\begin{rem} For any non-quasianalytic weight function $\omega$ such that $\log(1+t)=o(\omega(t))$ for $t\to\infty$, the space $(\cO'_{M,\omega}(\R^N),\star)$ is not a convolution topological algebra (see Section 4).
	
\end{rem}

\section{Hypocontinuity and discontinuity}
\subsection{Hypocontinuity}
In this final section we discuss the hypocontinuity of the multiplication mapping on some pairs between the spaces $\cO_{M,\omega}(\R^N)$, $\cO_{C,\omega}(\R^N)$, $\cS_{\omega}(\R^N)$ and their duals. 

Let us recall that 
if $E$, $F$ and $G$ are  topological vector spaces and $b\colon  E\times F\to G$  is bilinear map, then
$b$ is called \textit{hypocontinuous} if the following holds:

(1) For every bounded subset $A$ of $E$, the set $\{b(x, \cdot)\colon x\in A\}$ is equicontinuous
in $\cL(F,G)$;

(2) For every bounded subset $B$ of $F$, the set $\{b(\cdot, y)\colon  y\in B \}$ is equicontinuous
in $\cL(E,G)$.

\begin{prop}\label{P.Mult1} Let $\omega$ be a non-quasianalytic weight function. Then the multiplication operator $M\colon \cO_{M,\omega}(\R^N)\times \cS_\omega(\R^N)\to \cS_\omega(\R^N)$, $(f,g)\mapsto fg$, is separately continuous and hence, a hypocontinuous bilinear mapping.
	\end{prop}

\begin{proof} By \cite[Theorem 4.4]{AC}  the operator $M_f:=M(f, \cdot) \colon \cS_\omega(\R^N)\to \cS_\omega(\R^N)$, is continuous for all $f\in \cO_{M,\omega}(\R^N)$.

	Let $g\in \cS_\omega(\R^N)$ be fixed. We claim that  the operator $M_g:=M( \cdot,g) \colon \cO_{M,\omega}(\R^N)\to \cS_\omega(\R^N)$, is continuous. To show the claim, we suppose that $\{f_i\}_i\subset  \cO_{M,\omega}(\R^N)$ is  any net such that $f_i\to f$ in $\cO_{M,\omega}(\R^N)$ and $M_g(f_i)=f_ig\to h$ in $\cS_\omega(\R^N)$. Then by  \cite[Theorem 5.2(1) and Proposition 5.6]{AC} we have that $f_i\to f$ in $\cE_{\omega}(\R^N)$ and hence, $M_g(f_i)=f_ig\to fg$ in $\cE_{\omega}(\R^N)$ too. But
	$M_g(f_i)=f_ig\to h$ also in $\cE_\omega(\R^N)$. Therefore, $h=fg=M_g(f)$. Since $\{f_i\}_i\subset  \cO_{M,\omega}(\R^N)$ is arbitrary, this shows that the graph of the operator $M_g$ is closed. But, the space $\cO_{M,\omega}(\R^N)$ is ultrabornological by \cite{De}, hence barrelled, and $\cS_\omega(\R^N)$ is a Fr\'echet space, and so $M_g$ is necessarily continuous.
	
	Finally, as the operator $M$ is separately continuous and $\cO_{M,\omega}(\R^N)$ and $\cS_\omega(\R^N)$ are barrelled spaces, applying \cite[Theorem 41.2]{Tr}  we can conclude that $M$ is a hypocontinuous bilinear mapping. 
\end{proof}

\begin{cor}\label{C.Mult1}  Let $\omega$ be a non-quasianalytic weight function. Then the multiplication operator $M\colon \cO_{C,\omega}(\R^N)\times \cS_\omega(\R^N)\to \cS_\omega(\R^N)$, $(f,g)\mapsto fg$, is separately continuous and hence, a hypocontinuous bilinear mapping.
	\end{cor}

\begin{proof} By  \cite[Theorem 3.8(1)]{AC} the space $\cO_{C,\omega}(\R^N)$ is continuously included in $\cO_{M,\omega}(\R^N)$. So, via Proposition  \ref{P.Mult1} it follows that the multiplication operator $M$ is separately continuous and so hypocontinuous.
	\end{proof}

\begin{prop}\label{P.Mult2} Let $\omega$ be a non-quasianalytic weight function. Then the multiplication operator $M\colon \cO_{M,\omega}(\R^N)\times \cS'_\omega(\R^N)\to \cS'_\omega(\R^N)$, $(f,T)\mapsto fT$, is separately continuous and hence, a hypocontinuous bilinear mapping.
	\end{prop}

\begin{proof} By \cite[Theorem 4.6]{AC} the multiplication operator $M_f:=M(f,\cdot) \colon \cS'_\omega(\R^N)\to  \cS'_\omega(\R^N)$ is continuous for all $f\in \cO_{M,\omega}(\R^N)$. 
	
	Let $T\in \cS'_\omega(\R^N)$ be fixed. We claim that the operator $M_T:=M(\cdot,T)\colon \cO_{M,\omega}(\R^N)\to \cS'_\omega(\R^N)$ is continuous. To show the claim, we fix a $0$-neighborhood $V$ in $\cS'_\omega(\R^N)$. We can suppose that $V=\{S\in \cS'_\omega(\R^N)\colon \sup_{g\in B}|\langle S,g\rangle|\leq \varepsilon \}$ for some $\varepsilon >0$ and a bounded subset $B$ of  $\cS'_\omega(\R^N)$. On the other hand, as $T\in \cS'_\omega(\R^N)$,  there exist $m,n\in\N$ and $c>0$ such that $|\langle T,g\rangle|\leq c\,q_{m,n}(g)$ for all $g\in \cS_\omega(\R^N)$. We now set $U:=\{g\in \cS_\omega(\R^N)\colon q_{m,n}(g)\leq c^{-1}\varepsilon\}$ which is a $0$-neighborhood in $\cS_\omega(\R^N)$,  and define the set 
	\[
	W:=\{T\in \cL(\cS_\omega(\R^N))\colon T(B)\su U\}.
	\]
	Then $W$ is a $0$-neighborhood  in $\cL_b(\cS_\omega(\R^N))$. Since  $\cO_{M,\omega}(\R^N)$ is a subspace of $\cL_b(\cS_\omega(\R^N))$, there exists a  $0$-neighborhood $W_0$ in $\cO_{M,\omega}(\R^N)$ for which 
	$fg\in U$ for all $f\in W_0$ and $g\in B$ (i.e., $M_f(B)\su U$ for all $f\in W_0$). Accordingly, $M_T(W_0)\su V$. Indeed, for a fixed $f\in W_0$, $fg\in U$ for all $g\in B$ and hence, $q_{m,n}(fg)\leq c^{-1}\varepsilon$ for all $g\in B$. This yields for all $g\in B$ that
	\[
	|\langle M_T(f),g\rangle|=|\langle fT,g\rangle|=|\langle T,fg\rangle|\leq c\,q_{m,n}(fg)\leq \varepsilon.
	\]
This means that $M_T(f)\in V$. Since $M_T(W_0)\su V$, it is obviuous that $W_0\su M_T^{-1}(V)$. Finally, since $V$ is an arbitrary $0$-neighborhood in $\cS'_\omega(\R^N)$, we can conclude that the operator $M_T$ is continuous.	

We now observe that the space $\cS_\omega(\R^N)$ is nuclear  and hence distinguished, i.e., its strong dual $\cS'_\omega(\R^N)$ is a barrelled lcHs. So, since $M$ is separately continuous and $\cO_{M,\omega}(\R^N)$ and $\cS'_\omega(\R^N)$ are barrelled lcHs, applying \cite[Theorem 41.2]{Tr}  we get that $M$ is a hypocontinuous bilinear mapping.
	\end{proof}

\begin{cor}\label{C.Mult2}
Let $\omega$ be a non-quasianalytic weight function satisfying the condition $\log(1+t)=o(\omega(t))$ as $t\to\infty$. Then the multiplication operator $M\colon \cO_{M,\omega}(\R^N)\times \cO'_{C,\omega}(\R^N)\to \cS'_\omega(\R^N)$, $(f,T)\mapsto fT$, is separately continuous and hence, a hypocontinuous bilinear mapping.	
\end{cor}

\begin{proof} By  \cite[Theorems 3.8(2) and Theorem 3.9]{AC} the space $\cO'_{C,\omega}(\R^N)$ is continuously included in $\cS'_{\omega}(\R^N)$. So, via Proposition  \ref{P.Mult2} it follows that the multiplication operator $M$ is separately continuous and  so hypocontinuous, being   $\cO'_{C,\omega}(\R^N)$ topologically  isomorphic to $\cO_{M,\omega}(\R^N)$ via the Fourier transfom and hence barrelled.
\end{proof}

\begin{prop}\label{P.Multi3} Let $\omega$ be a non-quasianalytic weight function. Then the multiplication operator $M\colon \cO_{M,\omega}(\R^N)\times \cO'_{M,\omega}(\R^N)\to \cO'_{M,\omega}(\R^N)$, $(f,T)\mapsto fT$, is separately continuous and hence, a hypocontinuous bilinear mapping.
	\end{prop}

\begin{proof} By Corollary \ref{C.Moltiplicatori} the multiplication operator $M_f:=M(f,\cdot)\colon \cO'_{M,\omega}(\R^N)\to \cO'_{M,\omega}(\R^N)$ is continuous for all $f\in \cO_{M,\omega}(\R^N)$. 
	
	Let $T\in \cO'_{M,\omega}(\R^N)$ be fixed. We claim that  $M_T:=M(\cdot, T)\colon \cO_{M,\omega}(\R^N)\to \cO'_{M,\omega}(\R^N)$ is a continuous operator. To this end, we observe that there exist a function $k\in \cS_\omega(\R^N)$, $m\in\N$ and $c>0$ such that for all $g\in \cO_{M,\omega}(\R^N)$ we have
	\begin{equation}\label{eq.dis1}
	|\langle T, g\rangle|\leq c\,q_{m,k}(g).
	\end{equation}
If $B$ is any bounded subset of $\cO_{M,\omega}(\R^N)$, then from \eqref{eq.dis1} it follows  for all $f\in \cO_{M,\omega}(\R^N)$ that
	\begin{equation}\label{eq.dis2}
	\sup_{g\in B}|M_T(f)(g)|=\sup_{g\in B}|\langle T, fg\rangle|\leq c\,\sup_{g\in B}q_{m,k}(fg).
	\end{equation}
	Now, as it is shown in Theorem \ref{mu} there exist $l\in \cS_\omega(\R^N)$ and  $m'\in N$ such that $q_{m,k}(uv)\leq C q_{m',l}(u)q_{m',l}(v)$ whenever $u,v\in \cS_\omega(\R^N)$  and for $C=e^{mL}$. Accordingly, we obtain via \eqref{eq.dis2} that
	\[
	\sup_{g\in B}|M_T(f)(g)|\leq Cc\,\sup_{g\in B}q_{m',l}(g)q_{m',l}(f)
	\]
	for all $f\in \cO_{M,\omega}(\R^N)$, where $D:=\,\sup_{g\in B}q_{m',l}(g)<\infty$, being   $B$  a bounded subset of $\cO_{M,\omega}(\R^N)$. Since $B$ is an arbitrary bounded subset of $\cO_{M,\omega}(\R^N)$, this yields that the multiplication operator $M_T\colon \cO_{M,\omega}(\R^N)\to \cO'_{M,\omega}(\R^N)$ is continuous.
	
	Finally, since $M$ is separateley continuous and  $\cO_{M,\omega}(\R^N)$ and $\cO'_{M,\omega}(\R^N)$ are barrelled lcHs (the latter space is barrelled because it is reflexive as the strong dual of a Montel space), applying \cite[Theorem 41.2]{Tr} we get that $M$ is a hypocontinuous bilinear mapping.
		\end{proof}
	
	\begin{prop}\label{P.Multi4} Let $\omega$ be a non-quasianalytic weight function satisfying the condition $\log(1+t)=o(\omega(t))$ as $t\to\infty$. Then the multiplication operator $M\colon \cO_{C,\omega}(\R^N)\times \cO'_{C,\omega}(\R^N)\to \cO'_{C,\omega}(\R^N)$, $(f,T)\mapsto fT$, is separately continuous and hence, a hypocontinuous bilinear mapping.
	\end{prop}

\begin{proof} By Corollary \ref{C.Moltiplicatori} the multiplication operator $M_f:=M(f,\cdot)\colon \cO'_{C,\omega}(\R^N)\to \cO'_{C,\omega}(\R^N)$ is continuous for all $f\in \cO_{C,\omega}(\R^N)$. 
	
	Let $T\in \cO'_{C,\omega}(\R^N)$ be fixed. We show that  $M_T:=M(\cdot, T)\colon \cO_{C,\omega}(\R^N)\to \cO'_{C,\omega}(\R^N)$ is a continuous operator. To see this, let $\{f_i\}_i\subset \cO_{C,\omega}(\R^N)$ be any net such that $f_i\to f$ in $\cO_{C,\omega}(\R^N)$ and $M_T(f_i)=f_iT\to S$ in $\cO'_{C,\omega }(\R^N)$. Then $(f_iT)(g)=T(f_ig)\to S(g)$ for all $g\in\cO_{C,\omega }(\R^N)$. On the other hand, $f_ig\to fg$ in $\cO_{C,\omega}(\R^N)$ for all $g\in\cO_{C,\omega}(\R^N)$ by Proposition \ref{P.O_C-M}, thereby implying that $T(f_ig)\to T(fg)$ for all $g\in\cO_{C,\omega}(\R^N)$. Therefore, $T(fg)=S(g)$ for all $g\in\cO_{C,\omega}(\R^N)$. This means that $S=fT=M_T(f)$. Since $\{f_i\}_i\subset \cO_{C,\omega}(\R^N)$ is arbitrary, this shows that the graph of the operator $M_T$ is closed. But, the space $\cO_{C,\omega}(\R^N)$ is an (LF)-space and $\cO'_{C,\omega}(\R^N)$ is a webbed space, and so $M_T$ is necessarily continuous (see \cite{J}).

	Finally, since $M$ is separateley continuous and $\cO_{C,\omega}(\R^N)$ and $\cO'_{C,\omega}(\R^N)$ are barrelled lcHs, applying \cite[Theorem 41.2]{Tr} we get that $M$ is a hypocontinuous bilinear mapping.
\end{proof}

\subsection{Discontinuity}
In this last part we show examples of multiplication and convolution mapping on some pairs between the spaces $\cO_{M,\omega}(\R^N)$, $\cO_{C,\omega}(\R^N)$, $\cS_{\omega}(\R^N)$ and their duals that are not continuous.

\begin{prop}\label{disc2}
	Let $\omega$ be a non-quasianalytic weight function. Then the following assertions hold true:
	\begin{itemize}
		\item[(i)] The multiplication operator $M\colon \cO_{M,\omega}(\R^N)\times \cO'_{M,\omega}(\R^N)\to \cO'_{M,\omega}(\R^N)$, $(f,T)\mapsto fT$, is discontinuous;
		\item[(ii)] The multiplication operator $M\colon \cO_{C,\omega}(\R^N)\times \cO'_{C,\omega}(\R^N)\to \cO'_{C,\omega}(\R^N)$, $(f,T)\mapsto fT$, is discontinuous.
	\end{itemize}
\end{prop}

\begin{proof} Since the proof is analogous in both the cases,  we show only assertion (i).
	
	Assume by contradiction that $M\colon \cO_{M,\omega}(\R^N)\times \cO'_{M,\omega}(\R^N)\to \cO'_{M,\omega}(\R^N)$, $(f,T)\mapsto fT$, is continuous. Thus, the  canonical bilinear form $\langle \cdot,\cdot\rangle\colon \cO_{M,\omega}(\R^N)\times \cO'_{M,\omega}(\R^N)\to \C$, $(f,T)\mapsto \langle f,T\rangle:=T(f)$ is continuous, being it equals to the composition map $\mathbf{1}\circ M$ of the continuous mapping ${\mathbf{1}}\colon \cO'_{M,\omega}(\R^N)\to\C$, $T\mapsto  T(\mathbf{1})$ (note that the function $\mathbf{1}\in \cO_{M,\omega}(\R^N)$) and the multiplication $M$.   But the canonical bilinear form of a lcHs $E$ is continuous if, and only if, the space $E$ is normed (see \cite[Page 359]{H}), which is not the case for $\cO_{M,\omega}(\R^N)$.
	\end{proof}

\begin{prop}\label{disc1}
	Let $\omega$ be a non-quasianalytic weight function  satisfying the condition $\log(1+t)=o(\omega(t))$ as $t\to \infty$. Then the following properties hold true:
	\begin{itemize}
		\item[(i)] The convolution operator $C\colon \cO_{M,\omega}(\R^N)\times \cO'_{M,\omega}(\R^N)\to \cO_{M,\omega}(\R^N)$, $(f,T)\mapsto T\star f$, is discontinuous.
		\item[(ii)] The convolution operator $C\colon \cO_{C,\omega}(\R^N)\times \cO'_{C,\omega}(\R^N)\to \cO_{C,\omega}(\R^N)$, $(f,T)\mapsto T\star f$, is discontinuous.
	\end{itemize}
\end{prop}

\begin{proof} We first observe that the mapping are well defined by  Remark \ref{OssC}.
	
	 Since the proof is analogous in both the cases,  we show only property (i).
	
	 Assume by contradiction that $C\colon \cO_{M,\omega}(\R^N)\times \cO'_{M,\omega}(\R^N)\to \cO_{M,\omega}(\R^N)$, $(f,T)\mapsto T\star f$, is continuous. Therefore, the  canonical bilinear form $\langle \cdot,\cdot\rangle\colon \cO_{M,\omega}(\R^N)\times \cO'_{M,\omega}(\R^N)\to \C$, $(f,T)\mapsto \langle f,T\rangle:= T(f)$, is continuous, being it equal to  the composition map $\delta\circ R\circ C$  of the continuous mappings $\delta\colon \cO_{M,\omega}(\R^N)\to\C$, $f\mapsto f(0)$, $R\colon \cO_{M,\omega}(\R^N)\to \cO_{M,\omega}(\R^N)$, $f\mapsto \check{f}$, and the convolution $C$.  But the canonical bilinear form of a lcHs $E$ is continuous if, and only if, the space $E$ is normed (see \cite[Page 359]{H}), which is not the case for $\cO_{M,\omega}(\R^N)$.
	%%%%%%%%%%%%%%%%%%%%%%%%%%
	\end{proof}
	
\begin{rem}
	The proof of Proposition \ref{disc1} can be also achieved applying  Proposition \ref{disc2} and using the properties of the Fourier transform.
	% But to proceed in this way  we need to assume that the weight function $\omega$ has to satisfy the condition $\log(1+t)=o(\omega(t))$ as $t\to\infty$.
\end{rem}

Therefore, from Propositions \ref{P.Multi3}, \ref{P.Multi4} and \ref{disc2},  we get that the multiplication operators $M\colon \cO_{M,\omega}(\R^N)\times \cO'_{M,\omega}(\R^N)\to \cO'_{M,\omega}(\R^N)$ and $M\colon \cO_{C,\omega}(\R^N)\times \cO'_{C,\omega}(\R^N)\to \cO'_{C,\omega}(\R^N)$ are hypocontinuous but not continuous. The same also holds true for the hypocontinuous multiplication operator $M\colon \cO_{M,\omega}(\R^N)\times \cS'_\omega(\R^N)\to \cS'_\omega(\R^N)$, as we show in the following result.
\begin{prop}\label{disc3}
	Let $\omega$ be a non-quasianalytic weight function. Then the multiplication operator $M\colon \cO_{M,\omega}(\R^N)\times \cS'_\omega(\R^N)\to \cS'_\omega(\R^N)$, $(f,T)\mapsto fT$, is discontinuous.
\end{prop}
\begin{proof}
	Assume by contradiction that $M\colon \cO_{M,\omega}(\R^N)\times \cS'_\omega(\R^N)\to \cS'_\omega(\R^N)$ is continuous. Since $\cS_\omega(\R^N)$ is continuously included in $\cO_{M,\omega}(\R^N)$, it follows that  $M\colon \cS_{\omega}(\R^N)\times \cS'_\omega(\R^N)\to \cS'_\omega(\R^N)$ is also continuous. Arguing as in the proof of Propositon \ref{disc2}, we get that the canonical bilinear form of $\cS_\omega(\R^N)$ is continuous. This is a contradiction.
	\end{proof}

Using the properties of the Fourier transfom, we obtain the following result as a consequence.
\begin{cor}
	Let $\omega$ be a non-quasianalytic weight function satisfying the condition $\log(1+t)=o(\omega(t))$ as $t\to\infty$. Then the convolution operator $C\colon \cO'_{C,\omega}(\R^N)\times \cS'_{\omega}(\R^N)\to \cS'_{\omega}(\R^N)$, $(T,S)\mapsto T\star S$, is discontinuous.
\end{cor}

Now, we prove that $\left(\cO'_{M,\omega}(\R^N), \star\right)$ is not a topological algebra.

\begin{prop}\label{disc5}
	Let $\omega$ be a non-quasianalytic weight function satisfying the condition $\log(1+t)=o(\omega(t))$ as $t\to\infty$. Then the convolution operator $C\colon \cO'_{M,\omega}(\R^N)\times \cO'_{M,\omega}(\R^N)\to \cO'_{M,\omega}(\R^N)$, $(T,S)\mapsto T\star S$, is discontinuous.
\end{prop}

\begin{proof} We first observe that by Remark \ref{OssC} the convolution $C$ is well-defined. Now,
	assume by contradiction that  $C\colon \cO'_{M,\omega}(\R^N)\times \cO'_{M,\omega}(\R^N)\to \cO'_{M,\omega}(\R^N)$ is continuous. Since $\cO_{M,\omega}(\R^N)$ is continuously included in the space $\cO_M(\R^N)$ of multipliers of $\cS(\R^N)$ with dense range, $\cO'_M(\R^N)$ is continuously included in $\cO'_{M,\omega}(\R^N)$. Therefore, we get that the map $C\colon \cO'_{M}(\R^N)\times \cO'_{M}(\R^N)\to \cO'_{M,\omega}(\R^N)$ is continuous with range a subset of  $\cO'_{M}(\R^N)$. This yields that the map $\cO'_{M}(\R^N)\times \cO'_{M}(\R^N)\to \cO'_{M}(\R^N)$ has closed graph, as it is easy to verify. By applying  the closed graph theorem for (LF)-spaces (see, f.i., \cite[Chap. 5, 5.4.1]{J}) we obtain that $\cO'_{M}(\R^N)\times \cO'_{M}(\R^N)\to \cO'_{M}(\R^N)$ is continuous, after having observed that $\cO'_{M}(\R^N)$ in an (LF)-space. This is a contradiction with \cite[Proposition 4]{La}.
\end{proof}

Thanks to Proposition \ref{disc5} and the properties of the Fourier transform we get that $(\cO_{C,\omega}(\R^N),\cdot)$ is not a topological algebra.

\begin{cor}\label{disc4}
	Let $\omega$ be a non-quasianalytic weight function satisfying the condition $\log(1+t)=o(\omega(t))$ as $t\to\infty$. Then the multiplication operator $M\colon \cO_{C,\omega}(\R^N)\times \cO_{C,\omega}(\R^N)\to \cO_{C,\omega}(\R^N)$, $(f,g)\mapsto fg$, is discontinuous.
\end{cor}

{\bf Acknowledgement.} The authors would like to thank Andreas Debrouwere for finding a gap in the earlier version of Theorem \ref{T.complete}.

\end{document}